\documentclass{amsart}

\usepackage[margin=1.5in]{geometry}
\usepackage{amsmath}
\usepackage{amssymb, amsthm, mathrsfs}
\usepackage{hyperref}

\newtheorem{theorem}{Theorem}[section]

\newtheorem{lemma}[theorem]{Lemma}
\newtheorem{proposition}[theorem]{Proposition}
\newtheorem{corollary}[theorem]{Corollary}

\theoremstyle{definition}
\newtheorem{definition}[theorem]{Definition}

\theoremstyle{remark}
\newtheorem{remark}[theorem]{Remark}

\numberwithin{equation}{section}

\newcommand{\tr}{\textup{tr}}
\newcommand{\Rm}{\textup{Rm}}
\newcommand{\Ric}{\textup{Ric}}
\newcommand{\Div}{\textup{div}}

\newcommand{\FS}{\textup{FS}}
\newcommand{\Tr}{\textup{Tr}}

\newcommand{\loc}{\textup{loc}}

\newcommand{\C}{\mathbb{C}}

\newcommand{\Z}{\mathbb{Z}}

\newcommand{\CP}{\mathbb{CP}}

\newcommand{\sL}{\mathscr{L}}

\newcommand{\dd}[1]{\frac{\partial}{\partial #1}}

\newcommand{\ddbar}{\partial \bar{\partial}}

\begin{document}
\title{Rotational Symmetry of Conical K\"ahler-Ricci Solitons}
\author{Otis Chodosh}
\address{Department of Mathematics, Stanford University}
\email{ochodosh@math.stanford.edu}

\author[Frederick Fong]{Frederick Tsz-Ho Fong}
\address{Department of Mathematics, Brown University}
\email{fong@math.brown.edu}

\date{\today}
\thanks{The first named author was supported in part by a National Science Foundation Graduate Research Fellowship DGE-1147470. He would like to thank Simon Brendle for many discussions concerning his soliton uniqueness results, as well as for his support and encouragement. The second named author would like to thank Richard Schoen, Simon Brendle and Yanir Rubinstein who aroused his interest concerning topics related to this paper when he was a graduate student at Stanford. He would also like to thank Nicos Kapouleas for discussions which motivated him to consider this problem.}
\begin{abstract}
We show that expanding K\"ahler-Ricci solitons which have positive holomorphic bisectional curvature and are $C^{2}$-asymptotic to K\"ahler cones at infinity must be the $U(n)$-rotationally symmetric expanding solitons constructed by Cao. 
\end{abstract}
\maketitle
\section{Introduction}

An expanding gradient Ricci soliton is a Riemannian manifold $(M,g)$ which satisfies $2\Ric +  g = 2 D^2 f$ for some smooth function $f$, called the soliton potential. Solitons provide local models for singularity formulation under the Ricci flow, and play a central role in the Hamilton--Perelman theory of Ricci flow. The fundamental object of study in this paper are expanding K\"ahler-Ricci solitons, which are expanding Ricci solitons that are also K\"ahler metrics. These solitons may be expressed in holomorphic coordinates as
\begin{equation}\label{eq:KRS}
2R_{i\bar{j}} + g_{i\bar{j}} = 2D_i D_{\bar{j}} f, \qquad D_i D_j f = 0.
\end{equation}
In particular, $\nabla f$ must be a real holomorphic vector field on $M$.

In addition to providing singularity models for the Ricci flow, solitons are also of interest for a variety of other reasons. For example, the soliton equation is a natural generalization of the Einstein equation, which has been the subject of a great deal of study in both the Riemannian and K\"ahler settings. Additionally, expanding solitons are related to the Harnack inequality for the Ricci flow, cf.\ \cite{Cao:KRFsolitons} for the relationship between the Harnack inequality for K\"ahler-Ricci flows with nonnegative holomorphic bisectional curvature and expanding K\"ahler-Ricci solitons. Also, the K\"ahler-Ricci solitons constructed in e.g.\ \cite{FeldmanIlmanenKnopf} provide interesting examples in which it is possible to continue K\"ahler-Ricci flow through a singular time. As such, they provide a model case to study the (yet unresolved) problem of constructing a good theory of weak solution to the Ricci flow, similar the theory for other geometric heat equations, e.g.\ Brakke flow for mean curvature flow. The interested reader might consult \cite[Theorem 1.6]{FeldmanIlmanenKnopf} and the introduction in \cite{FutakiWang:KRFsolitons}.

The (steady and expanding) K\"ahler-Ricci solitons constructed on $\C^{n}$ by Cao in \cite{Cao:ExistenceofKRFsol,Cao:KRFsolitons} and (shrinking and expanding) solitons on complex line bundles over $\CP^{n}$ constructed by Feldman--Ilmanen--Knopf in \cite{FeldmanIlmanenKnopf} are rotationally symmetric. The expanding solitons constructed in these works may be seen to be solutions of Ricci flow with initial conditions (in the Gromov-Hausdorff sense) a K\"ahler cone on $\C^{n}\backslash\{0\}$ and $(\C^{n}\backslash\{0\})/\Z_{k}$ respectively, see \cite{FeldmanIlmanenKnopf} for further discussion concerning this point. These constructions have been generalized in several works, including \cite{DancerWang:CohomogeneityOne} and \cite{FutakiWang:KRFsolitons} to produce K\"ahler-Ricci solitons on holomorphic bundles over K\"ahler-Einstein manifolds via various cohomogeneity-one ansatzes. The interested reader may refer to the survey \cite{Cao:RecentSoliton} concerning recent progress in the study of Ricci solitons.

In this paper, we study expanding K\"ahler-Ricci solitons with positive bisectional curvature, under the assumption that they are asymptotically conic at infinity. Complete K\"ahler manifolds with positive (and non-negative) holomorphic bisectional curvature have been extensively studied. There is a well known conjecture, known as the Yau's Uniformization Conjecture, which says that any complete non-compact K\"ahler manifold with positive holomorphic bisectional curvature must be biholomorphic to $\C^n$. Many interesting results verify that this conjecture is true under certain geometric and analytic assumptions (see the survey paper \cite{ChauTam:YauConjSurvey} and the references therein). In particular, it was shown by Chau--Tam in \cite[Corollary 1.1]{ChauTam:OnComplexStr} (building on the work of Ni in \cite{Ni:AncientSol}) that any complete noncompact K\"ahler manifold with nonnegative and bounded holomorphic bisectional curvature and of maximal volume growth must be biholomorphic to $\C^n$. The maximal volume growth hypothesis clearly holds if the metric has a $U(n)$-rotationally symmetric tangent cone at infinity (in the Gromov--Hausdorff sense) and hence such a manifold must be biholomorphic to $\C^{n}$.

In order to state the precise conical asymptotic assumptions on the expanding K\"ahler-Ricci solitons we are considering here, we first define the $U(n)$-rotationally symmetric \emph{K\"ahler cone metrics} by $g_\alpha := 2 \textup{Re}\left(\ddbar |z|^{2\alpha}\right)$ (we discuss these metrics further in Section \ref{sect:kahler-cone}). We also define $\rho_{\lambda}:\C^{n}\to \C^{n}$ to be the \emph{dilation by $\lambda>0$ map}, i.e.\ $\rho_{\lambda}(z) = \lambda z$. Given these definitions, we may now define the asymptotic assumptions which we will consider in this paper:

\begin{definition}\label{defi:asym-cone}
A K\"ahler manifold $(M^{2n},g)$ is \emph{asymptotically conical} with cone angle $2\pi\alpha\in (0,2\pi)$ if there is a biholomorphism $F: \C^{n}\backslash K_{1} \to M^{2n}\backslash K_{2}$ (for $K_{1},K_{2}$ compact sets) so that 
\begin{equation*}
\lim_{\lambda\to\infty}\lambda^{-2\alpha} \rho_{\lambda}^{*}(F^{*}g) = g_{\alpha}
\end{equation*} 
in $C_{\loc}^{2}(\C^{n}\backslash K_{1},g_{\alpha})$.
\end{definition}

Our main result is:
\begin{theorem}\label{thm:main}
Suppose, for $n\geq 2$, that $(M^{2n}, g, f)$ is an expanding gradient K\"ahler-Ricci soliton with positive holomorphic bisectional curvature which is asymptotically conical in the sense of Definition \ref{defi:asym-cone}. Then, $(M,g,f)$ is isometric to one of the $U(n)$-rotationally symmetric expanding gradient solitons on $\C^{n}$, as constructed by Cao.
\end{theorem}

We remark that the expanding K\"ahler-Ricci solitons constructed by Cao in \cite{Cao:KRFsolitons}, are asymptotically conical as solitons and have positive bisectional curvature. Furthermore, they are the unique (up to isometry) $U(n)$-invariant expanding solitons asymptotic to a cone with a fixed cone angle (see \cite[Theorem 1.3]{FeldmanIlmanenKnopf}). We also remark that in (complex) dimension one, gradient Ricci solitons are completely classified: they are all rotationally symmetric as seen by considering the Killing vector $J\nabla f$ (where $J$ is some compatible complex structure). For example, there is exactly one expanding solitons with positive curvature, as discussed in \cite{Kotschwar:rotsym2dexpand}. See also the recent work \cite{BernsteinMettler:2dsol} and the references contained within for a complete classification.

Some context for the above result may be found in the recent works of Schulze--Simon and Cabezas-Rivas--Wilking \cite{SchulzeSimon:Expand-Cones,Cabezas-Rivas:RFviaCGexhaust}. In particular, by \cite[Remark 7.3]{Cabezas-Rivas:RFviaCGexhaust}, given a manifold with nonnegative complex sectional curvature (in particular, this holds for a manifold with nonnegative holomorphic bisectional curvature), one may construct a smooth Ricci flow with initial conditions (in the Gromov--Hausdorff sense) the tangent cone at infinity of this manifold. There is a well established theory of uniqueness for Ricci flow with bounded curvature when the initial metric a complete metric of bounded curvature (cf.\ \cite{ChenZhu:RFunique} and also \cite{Chen:RFstrongUnique} for uniqueness for three manifolds without the bounded curvature condition on the flow, but with the assumption that the initial metric has nonnegative, bounded Ricci curvature). However, very little is known (except in dimension two, see \cite{GiesenTopping:surfIncompRF}) concerning uniqueness of the Ricci flow with initial metrics which are not smooth/not complete, as in the examples constructed by Schulze--Simon and Cabezas-Rivas--Wilking. 

Our argument is inspired by the recent works of Brendle \cite{Brendle:3DSolitonUniqueness,Brendle:HighDimSoliton}, establishing $O(n)$-symmetry of certain steady solitons. In particular, Brendle was able to resolve a question posed by Perelman in his first paper \cite{Perelman:EntropyFunct}, concerning rotational symmetry of steady solitons in three dimensions:
\begin{theorem}[S. Brendle {\cite{Brendle:3DSolitonUniqueness}}] If $(M^{3},g,f)$ is a $\kappa$-noncollapsed, steady, gradient Ricci soliton, then it must be the $O(3)$-rotationally symmetric Bryant soliton constructed in \cite{Bryant:Solitons}. 
\end{theorem}

Brendle has also shown in \cite{Brendle:HighDimSoliton} that in dimensions greater than three, a steady soliton with positive sectional curvature which (parabolically) blows down to shrinking cylinder must be $O(n)$-rotationally invariant. One of the main ideas in \cite{Brendle:3DSolitonUniqueness,Brendle:HighDimSoliton} is to use the soliton equations to perturb approximate Killing vectors into actual Killing vectors. The approximate Killing vectors come from an assumption about the asymptotic geometry of the soliton (these assumptions are always satisfied in three dimensions, as long as the soliton is $\kappa$-noncollapsed). 

The approximate Killing vector technique of Brendle was subsequently used by the first named author in \cite{Chodosh:ExpandBrySol} to show that expanding (Riemannian) Ricci solitons which are asymptotic in a certain sense to the $O(n)$-rotationally symmetric conical metric $h_{\alpha} = dr^{2} + r^{2}(1-\alpha) g_{S^{n-1}}$ must be $O(n)$-rotationally invariant themselves. We note that such $O(n)$-rotationally symmetric solitons asymptotic to these cones have been constructed by Bryant in \cite{Bryant:Solitons}, see also \cite[Appendix A]{Chodosh:ExpandBrySol}. A similar result was obtained independently by Deruelle in \cite{Deruelle:ExpandBrySol}. 

However, several new arguments are needed to handle the K\"ahler setting in this paper. In particular, one interesting feature of the present work is that instead of perturbing the approximate Killing vectors to become actual Killing vectors, we show that the original Killing vector fields of the rotationally symmetric K\"ahler cone must in fact be Killing vectors on the soliton as well, by a combination of a Liouville type argument and a barrier argument for a Lichnerowicz PDE. The latter is a barrier argument in the vein of \cite[Proposition 5.1]{Chodosh:ExpandBrySol}, but we remark that it is also somewhat different because instead of using positive sectional curvature, we must use the positivity of holomorphic bisectional curvature (which is a weaker condition than positive sectional curvature). We finally note that thanks to the strength of our Liouville type theorem, we are also able to weaken somewhat the asymptotic conditions as compared with \cite{Chodosh:ExpandBrySol}. 

The paper is organized as follows. We first give relevant definitions and background material in Section \ref{sec:prelim}, including a discussion of K\"ahler cones and holomorphic bisectional curvature. In Section \ref{sec:refine-asymp}, we show that the asymptotically conic assumption implies a refined asymptotic statement. In particular, we show that if Definition \ref{defi:asym-cone} holds, then the soliton potential function is controlled in a $C^{2}$-sense and we may perturb the coordinate system (while still preserving the asymptotically conical condition) so that the soliton potential achieves its minimum at the origin. This will later prove crucial for the proof, as then $\nabla f$ and the Killing vectors on $\C^{n}$ corresponding to rotational symmetry will both vanish at the origin. 

In Section \ref{sec:PDES}, we present the key ingredients which will be used to establish the main result. We begin by recalling that if a vector field $U$ satisfies $\Delta U + D_{X}U -\frac 12 U=0$ (here $X=\nabla f$ is the soliton vector field) on an expanding Ricci soliton, then the Lie derivative $h=\mathscr{L}_{X}g$ satisfies the Lichnerowicz PDE, $\Delta_{L}h + \mathscr{L}_{X}h - h =0$. We show, using a Liouville type argument that if a real holomorphic vector field approximately satisfies the first PDE and if it vanishes in the same place as $X$, then it must in fact satisfy the PDE exactly. We also show, using $2\Ric +g$ as a barrier, that solutions to the Lichnerowicz PDE which decay at infinity must vanish identically. 

The proof of the main theorem is given in Section \ref{sec:proof-main-res}. The main idea of the proof is that, after ``centering'' the asymptotically conical coordinate system, we may take vectors fields on $\C^{n}$ corresponding to the Killing vectors generating rotations of the K\"ahler cone. Then, pushing them forward to the soliton, the asymptotics will imply that they are approximate Killing vectors and that they approximately satisfy the first PDE discussed above. Then, by the results of Section \ref{sec:PDES}, we first see that they must satisfy the vector field PDE exactly. Furthermore, using our results concerning solutions to the Lichnerowicz PDE, we see that they are exact Killing vectors for the soliton. 

Finally, we include a brief discussion concerning shrinking solitons in Section \ref{sec:shrink}. 

\section{Preliminaries}\label{sec:prelim}
In this section, we introduce relevant definitions and background material. We first remark that in the remainder of the of the work it will be convenient to define $X = \nabla f$ to be the soliton vector field. We remark that because the Ricci tensor $\Ric$ and the metric $g$ are of $(1,1)$-type, the soliton equation \eqref{eq:KRS} forces $X$ to be real holomorphic (recall that a vector field $U$ is real holomorphic if the $(1,0)$ part of $U$ has holomorphic coefficients).

\subsection{K\"ahler cones}\label{sect:kahler-cone}
For any $\alpha \in (0,1)$, denote by $g_\alpha$ the \emph{K\"ahler cone} metric with cone angle $2\pi\alpha$, i.e.
$$g_\alpha = 2 \textup{Re}\left(\ddbar |z|^{2\alpha}\right),$$
where $|z|$ is the Euclidean norm of $z = (z_1, \ldots, z_n) \in \C^n \backslash \{0\}$.

The cone metric is a warped product of $(0,\infty)$ and $S^{2n-1}$ with a Berger's sphere metric (we emphasize that this is \emph{not} the family of $O(n)$-invariant cones considered in e.g.\ \cite{Chodosh:ExpandBrySol} as the $O(n)$-invariant cones have cross sections of constant sectional curvature, in contrast with the $U(n)$-invariant K\"ahler cones just defined). The metric $g_{\alpha}$ may be written in the equivalent forms
\begin{align*}
g_\alpha & = dr^2 + r^2\left(\frac{\alpha^2}{4}d\theta^2 + \frac{\alpha}{4}\pi^*g_\FS\right)\\
& = dr^2 + (dr \circ J)^2 + \frac{\alpha r^2}{4}\pi^*g_\FS
\end{align*}
Here $S^1 \to S^{2n-1} \xrightarrow{\pi} \CP^{n-1}$ is the Hopf fibration map and $\theta$ denotes the $S^1$-fiber direction. $J$ is the unique complex struture such that $g_\alpha$ is $J$-Hermitian. Furthermore, $g_\FS$ is the Fubini-Study metric of $\CP^{n-1}$. The group $U(n)$ acts on these K\"ahler cones in manner which preserves the norm $|z|$. 
\begin{remark} We note that the radial distance function $r$ with respect to $g_{\alpha}$ is equal to $|z|^\alpha$. In particular, we caution the reader that $g_{\alpha}$ is \emph{not} uniformly equivalent to the flat metric $\delta$ on $\C^{n}$. In particular, when we use $r$ in the remainder of the work, we will \emph{only} mean the radial distance function with respect to $g_{\alpha}$. 
\end{remark}

To refine this remark, we note that in standard $\{z_i\}$ coordinates, the cone metric can be expressed as
$$(g_\alpha)_{i\bar{j}} = \alpha |z|^{2(\alpha-1)}\left(\delta_{ij} + (\alpha-1)|z|^{-2}\bar{z}_iz_j\right).$$
The eigenvalues of $(g_\alpha)_{i\bar{j}}$ are 
$$\alpha^2|z|^{2(\alpha-1)}, \alpha|z|^{2(\alpha-1)}, \ldots, \alpha|z|^{2(\alpha-1)}$$
which tend to $0$ as $|z| \to \infty$.

This is coherent with our remark above, and we see that if we have control of the norm of some quatnity, the norm taken with respect to one of the metrics $\delta$ and $g_{\alpha}$, then we may also control the norm with respect to the other (with a different rate of decay, corresponding these eigenvalues). We will make use of this in the proof of Lemma \ref{lem:liouville}. On the other hand, $F^*g$ and $g_\alpha$ are uniformly equivalent by the asymptotically conic condition, as established in Corollary \ref{coro:unwrap-conical}.

\subsection{Holomorphic bisectional curvature}
A K\"ahler manifold is a triple $(M,J,g)$ where $J$ is the complex structure and $g$ is a $J$-invariant metric such that $J$ is parallel with respect to $g$. On a K\"ahler manifold, one can define the holomorphic bisectional curvature $K_\C$ as
$$K_\C (X, Y) = \frac{\Rm(X, JX, Y, JY)}{|X|^2|Y|^2 - g(X, Y)^2}$$
for any $X, Y \in TM$. On the complexified tangent bundle $T_\C M = T^{1,0}M \oplus T^{0,1}M$, if one denotes
\begin{align*}
U & = \frac{1}{2}\left(X - \sqrt{-1}JX\right) \in T^{1,0}M\\
V & = \frac{1}{2}\left(Y - \sqrt{-1}JY\right) \in T^{1,0}M,
\end{align*}
then it is not hard to see that
$$K_\C (X,Y) = \frac{\Rm(U, \overline{U}, \overline{V}, V)}{|U|^2|V|^2 - |g(U,\overline{V})|^2}.$$

We say $(M, J, g)$ has positive holomorphic bisectional curvature if $K_\C(X,Y) > 0$ for any non-zero $X, Y \in TM$, or equivalently, $\Rm(U, \overline{U}, \overline{V}, V) > 0$ for any $U, V \in T^{1,0}M$. Note that positive sectional curvature implies positive holomorphic bisectional curvature, but \textit{not} vice versa. Moreover, positive holomorphic bisectional curvature implies positive Ricci curvature, but again \textit{not} vice versa. 

It was proved by Siu-Yau in \cite{SiuYau:1980} that any compact K\"ahler manifold with positive bisectional curvature must be biholomorphic to $\CP^n$ (see also \cite{Mori:FrankelConj}). Moreover, it was shown by Bando \cite{Bando:1984} and Mok \cite{Mok:Uniformization} that positivity of holomorphic bisectional curvature is preserved under the K\"ahler-Ricci flow. For complete, non-compact K\"ahler manifolds with positive bisectional curvature, it is conjectured by Yau that any such manifold must be biholomorphic to $\C^n$. This conjecture remains unresolved, in spite of many partial results as discussed in the introduction to this work. 

\section{Refined Conical Asymptotics}\label{sec:refine-asymp}

In this section we consider an expanding K\"ahler-Ricci soliton with positive bisectional curvature, which is asymptotically conical in the sense of Definition \ref{defi:asym-cone}. It will be necessary for the analysis in the subsequent sections to show that the asymptotically conical condition implies several refined properties which are a priori more restrictive. We first show that we may extend the biholomorphism in Definition \ref{defi:asym-cone} inwards to give a global biholomorphism between $\C^{n}$ and $M^{2n}$ in which the asymptotics are still valid.\footnote{We thank Yasha Eliashberg for explaining to us the proof of Lemma \ref{lem:bihol-extend}.} 

\begin{lemma}\label{lem:bihol-extend}
For $n\geq 2$, if $(M^{2n},g,f)$ is an expanding K\"ahler-Ricci soliton which is asymptotically conical and has positive holomorphic bisectional curvature, then we may extend the biholomorphism $F: \C^{n}\backslash K_{1}\to M^{2n}\backslash K_{2}$ into the compact set $K_{1}$, yielding $F:\C^{n}\to M^{2n}$ a biholomorphism still satisfying Definition \ref{defi:asym-cone}.
\end{lemma}
\begin{proof}
By the work of Chau--Tam and Ni (as discussed in the introduction) $M^{2n}$ is biholomorphic  to $\C^{n}$ (see \cite[Corollary 1.1]{ChauTam:OnComplexStr}). The issue here however, is whether we can find a biholomorphism $F$ which still has the desired conical asymptotic properties. Fixing global holomorphic coordinates $\{z^{i}\}_{i=1}^{n}$ on $M^{2n}$, we may consider the components of the map $F: \C^{n}\backslash K_{1}\to M^{2n}\backslash K_{2}$ with respect to these coordinates $F=(F_{1},\dots, F_{n})$. Now (enlarging the compact sets if necessary so that their complement is connected) Hartogs' extension theorem (which applies because we have assumed that $n\geq 2$) guarantees that the $F_{i}$ extend inward giving holomorphic maps $\hat F_{i}: \C^{n}\to M^{2n}$. 

Now, we claim that  $\hat F = (\hat F_{1},\dots, \hat F_{n})$ is a biholomorphism. First of all, the set where $D\hat F$ does not have maximal rank is a holomorphic subvariety of codimension $1$. Such a set, if non-empty, cannot be contained in a compact set, e.g.\ $K_{1}$ by the maximum principle. Thus we see that $D \hat F$ has maximal rank in all of $\C^{n}$, so it must be an immersion and thus a covering map. Finally, there must be points in $M^{2n}\backslash K_{2}$ with a single preimage, and thus $\hat F$ is a bijection. 
\end{proof}

By the previous lemma, we can (and will) assume, without loss of generality, that $K_{1}=K_{2}=\emptyset$ in Definition \ref{defi:asym-cone}. We now show that we may precompose the biholomorphism from Definition \ref{defi:asym-cone} by any shifting automorphism of $\C^{n}$ without changing any asymptotic conditions stated in Corollary \ref{coro:unwrap-conical}. This will allow us to construct a biholomorphism ${F} : \C^n \to M$ such that ${F}(0)$ is the (unique) critical point of $f$ in $M$. This will be crucial in the proof of Lemma \ref{lem:liouville}. Given any point $p \in \C^n$, we define the \emph{shift map} $\Lambda_p : \C^n \to \C^n$ by $\Lambda_p (z) = z + p$.

\begin{lemma}\label{lma:centering}
Shifting the coordinate system via $\Lambda_{p}$ preserves the asymptotically conical condition. More precisely, if $\tilde F:\C^{n}\to M$ is a biholomorphism so that Definition \ref{defi:asym-cone} is satisfied, then by defining $F : = \tilde{F} \circ \Lambda_{p}$, we have that $F$ is a biholomorphism still satisfying Definition \ref{defi:asym-cone}, i.e.\ the metric $F^{*}g$ satisfies 
\begin{equation*}
\lim_{\lambda\to\infty} \lambda^{-2\alpha} \rho_{\lambda}^{*}(F^{*}g)) = g_{\alpha}
\end{equation*}
in $C^{2}_{\loc}(\C^{n}\backslash \{0\},g_{\alpha})$

\end{lemma}
\begin{proof}
We first claim that for $j=0,1,2$, $r^{j}|\nabla^{j}(\Lambda_{p}^{*}g_{\alpha}-g_{\alpha})| = o(1)$ as $r \to \infty$. To see this, first observe that
\begin{equation}\label{eq:shift_g}
g_\alpha^{-1} (\Lambda_{p}^{*}g_{\alpha}-g_{\alpha}) = g_\alpha^{-1} (\Lambda_p^* g_\alpha) - \text{Id}.
\end{equation}
The eigenvalues of the endomorphism $g_\alpha^{-1} (\Lambda_p^* g_\alpha)$ with respect to the standard $z^k$-coordinate basis are
$$\frac{|z+p|^{2(\alpha-1)}}{|z|^{2(\alpha-1)}}, \ldots, \frac{|z+p|^{2(\alpha-1)}}{|z|^{2(\alpha-1)}}$$
which tend uniformly to 1 as $|z| \to \infty$. Thus, the above claim holds for $j=0$. Furthermore, by differentiating \eqref{eq:shift_g} and using the fact that $\Lambda_p$ is linear, one can verify that in fact the claim holds when $j = 1, 2$ as well. The assertion follows easily from this fact. 
\end{proof}

As a result of the the above lemma, we may ``center'' the biholomorphism $F$ in the asymptotically conical condition, i.e.\ by an appropriate shift, we have that $F(0)$ is the (unique) critical point of the potential function $f$ (there a unique critical point due to the fact that it the potential is strongly convex as seen from the soliton equation in combination with non-negativity of the Ricci curvature). We now identify the asymptotic behavior of the soliton potential function.

\begin{lemma}\label{lem:solpot-asymp}
If $(M,g,f)$ is an expanding K\"ahler-Ricci soliton which asymptotically conical, then the soliton potential function satisfies
\begin{equation*}
\lim_{\lambda\to\infty}\lambda^{-2\alpha} f\circ F \circ \rho_{\lambda} = \frac{r^{2}}{4}
\end{equation*}
in $C^{2}_{loc}(\C^{n}\backslash \{0\},g_{\alpha})$.
\end{lemma}

\begin{proof}
By, e.g.\ \cite[Lemma 2.2.3]{ChenDerulle:StructInftyExpandSol}, we know that $\left|f - \frac{d^2}{4}\right| \leq O(1)$ where $d$ is the distance function on $M$ from the minimum point of $f$. The conical asymptotics imply that $d \circ F = r + o(r)$, so we may write $\frac{d^{2}}{4} \circ F = \frac{r^{2}}{4} + \varphi$ for some $\varphi=o(r^{2})$. As such, we see that 
\begin{equation*}
\left| f \circ F - \frac{r^{2}}{4} - \varphi \right| \leq O(1).
\end{equation*}
Pulling back by $ \rho_{\lambda}$ and rescaling this yields
\begin{equation*}
\left| \lambda^{-2\alpha} f \circ F \circ \rho_{\lambda} - \frac{r^{2}}{4}  - \lambda^{-2\alpha}\varphi\circ\rho_{\lambda} \right| \leq C \lambda^{-2\alpha}
\end{equation*}
By the above asymptotic behavior of $\varphi$, we thus have established $C^{0}_{loc}$ convergence. Now, we establish the convergence of the higher derivatives. First, recall that the soliton equation may be written as $2\Ric_{g} + g = D_{g}^{2}f$. Pulling back by $F \circ \rho_{\lambda}$ and rescaling by $\lambda^{-2\alpha}$ yields (using the fact that the Ricci tensor does not scale under scalings of the metric)
\begin{equation}\label{eq:potent-rescale-sol}
2\lambda^{-2\alpha} \Ric_{\lambda^{-2\alpha} \rho_{\lambda}^{*}(F^{*}g)} + \lambda^{-2\alpha} \rho_{\lambda}^{*}(F^{*}g) = 2 D^{2}_{\lambda^{-2\alpha} \rho_{\lambda}^{*}(F^{*}g)}(\lambda^{-2\alpha} f\circ\rho_{\lambda}\circ F).
\end{equation}
By the asymptotically conical assumption, the left hand side of this expression converges as $\lambda \to \infty$. This establishes the convergence of $D^{2}_{g_{\alpha}} \lambda^{-2\alpha} f\circ F \circ \rho_{\lambda}$. Finally, rescaling Hamilton's identity $|\nabla f|^{2} +R = f$ (see \cite{Hamitlon:Singularities95}) in the same manner gives 
\begin{equation}\label{eq:potent-rescale-ham}
|\nabla (\lambda^{-2\alpha} f \circ  \rho_{\lambda} \circ F )|_{\lambda^{-2\alpha} \rho_{\lambda}^{*} (F^{*}(g))}^{2} +  R_{\lambda^{-2\alpha} \rho_{\lambda}^{*} (F^{*}(g))} = \lambda^{-2\alpha} (f\circ \rho_{\lambda}\circ F)
\end{equation}
which, in combination with the conical asymptotics, shows that at the very least the gradient of the rescaled potential is bounded. Given this, it is clear that the above convergence results imply the full claim.
\end{proof}

 The following corollary is a straightforward consequence of the above results, and is written in a form which is more convenient for later analysis. 
\begin{corollary}\label{coro:unwrap-conical}
The asymptotically conical condition, Definition \ref{defi:asym-cone}, implies that there is a biholomorphism $F: \C^{n}\to M$ so that $F(0)$ is a critical point of the potential function $f$ and moreover
\begin{enumerate}
\item the pulled back metric satisfies $F^{*}g = g_{\alpha} + k$ for some tensor $k$ satisfying $r^{j}|\nabla^{j}k| = o(1)$ as $r\to\infty$ for $j=0,1,2$, and
\item the soliton potential function satisfies $r^{j-2}\left| \nabla^{j}\left( f\circ F - \frac{r^{2}}{4} \right)\right| = o(1)$ as $r\to\infty$ for $j=0,1,2$.
\end{enumerate}
\end{corollary}

We will always choose such an $F$ in the subsequent sections.

\section{PDE's for Approximate Killing Vector Fields}\label{sec:PDES}
We first recall the following results from \cite{Chodosh:ExpandBrySol}, which are valid for any expanding soliton. Recall that $\Delta_{L}$ is the Lichnerowicz Laplacian, defined in any unitary frame $\{\eta_1, \ldots, \eta_n\}$ by
\begin{equation*}
\Delta_{L}h_{i\bar k}= \Delta h_{i\bar k} + 2 \Rm_{i\bar j \bar k l}h_{\bar l k} - \Ric_{i\bar{l}} h_{l \bar k}-\Ric_{l\bar k}h_{i\bar l}
\end{equation*}
for any $J$-invariant $(0,2)$-tensor $h$.

\begin{proposition}[{\cite[Proposition 3.1]{Chodosh:ExpandBrySol} }]\label{prop:lichPDE}
If a vector field $U$ satisfies $\Delta U + D_{X} U -\frac 12 U = 0$, then $h: = \mathscr{L}_{U}(g)$ satisfies 
\begin{equation*}
\Delta_{L}h + \sL_{X}h - h = 0.
\end{equation*}
\end{proposition}
This follows by direct computation, using the soliton equation. As a corollary of this, we have
\begin{corollary}[{\cite[Corollary 3.2]{Chodosh:ExpandBrySol}}]\label{cor:lichPDEforX}
The soliton vector field $X$ satisfies $\Delta X + D_{X}X -\frac 12 X= 0$ and thus $\sL_{X}(g) = 2\Ric + g$ satisfies
\begin{equation*}
\Delta_{L}(2\Ric+g) + \mathscr{L}_{X}(2\Ric + g) -(2\Ric + g)= 0.
\end{equation*}
\end{corollary}

In the remainder of this section, we will prove two results which will allow us to analyze solutions to the two PDEs just discussed. We will crucially use the assumption that $M$ is biholomorphic to $\C^n$ in order to apply the classical Liouville theorem in the following lemma. This will eventually play a key role in our proof that the approximate Killing vectors obtained from the asymptotically conic assumption are actually exact Killing vectors. 

\begin{lemma}\label{lem:liouville}
Suppose $U$ is a real holomorphic vector field on $M$ which vanishes at a critical point of $f$ in $M$. If the vector field $Q := \Delta U + D_X U - \frac{1}{2}U$ satisfies $|Q|_g = o(r)$, then necessarily $Q \equiv 0$.
\end{lemma}
\begin{proof}
It is not hard to see from the soliton equation \eqref{eq:KRS} that 
\begin{equation*}
Q = \Delta U + \Ric(U) - [X,U]
\end{equation*}
where $X=\nabla f$ is the soliton vector field (cf.\ \cite[Section 3]{Chodosh:ExpandBrySol}). Here $\Ric$ is regarded as an endomorphism on $TM$. Furthermore, using the fact that $X$ and $U$ are real holomorphic it is not hard to check that $[X,U]^{1,0} = [X^{1,0},U^{1,0}]$. Thus, by the K\"ahler condition (i.e. $J$ is parallel) we see that
\begin{equation*}
Q^{1,0} = \Delta U^{1,0} + \Ric(U^{1,0}) -[X^{1,0},U^{1,0}].
\end{equation*}
On the other hand, because $U$ is real holomorphic, a simple computation in local coordinates gives that 
\begin{align*}
\Delta U^{1,0} & = g^{l\bar{j}} (D_l D_{\bar{j}} + D_{\bar{j}} D_l ) \left(U^k \dd{z^{k}}\right)\\
& = g^{l\bar{j}} \dd{\bar{z}^{j}} \left(U^k \Gamma_{lk}^i \dd{z_i}\right) = - g^{l\bar{j}} \Ric_{k\bar{j}} U^k \dd{z^{l}}
\end{align*}
and therefore,
\begin{equation*}
\Delta U^{1,0} + \Ric(U^{1,0}) = 0.
\end{equation*}
As such, we see that $Q^{1,0} = [U^{1,0},X^{1,0}]$. In particular, this shows that if we write $F^*Q^{1,0} = Q^{k}\frac{\partial}{\partial z^{k}}$ (where $z^{k}$ are the coordinates fixed by Definition \ref{defi:asym-cone}) then the coefficients $Q^{k}$ are holomorphic functions on $\C^{n}$. We claim that the $Q^{k}$ have sublinear growth as holomorphic functions on $\C^{n}$. To see this, recall that the eigenvalues of $g_{\alpha}$ with respect to the standard $\{z^{k}\}$-coordinates are given by 
\begin{equation*}
\alpha^2|z|^{2(\alpha-1)}, \alpha|z|^{2(\alpha-1)}, \ldots, \alpha|z|^{2(\alpha-1)}.
\end{equation*}
Therefore, we have
\begin{align*}
\frac{1}{C} |z|^{2(\alpha-1)} |Q^k|^2 & \leq \frac{1}{C}|z|^{2(\alpha-1)}\sum_{j=1}^n |Q^j|^2\\
& \leq |(g_\alpha)_{i\bar{j}} Q^i Q^{\bar{j}}|\\
& = |Q|^2_{g_{\alpha}} = o(r^2) = o(|z|^{2\alpha})
\end{align*}
and so for each $k$, we have $|Q^k| \leq o(|z|)$. In other words, $Q^k$ has sublinear growth as a holomorphic function on $\C^{n}$. Thus, the classical Liouville's theorem implies $Q^k$ are constants for any $k$. Finally, we would like to conclude they vanish identically. To do so, it is enough to show that $Q^{k} = 0$ at some point in $\C^{n}$. However, this follows easily from the formula $Q^{1,0}=[U^{1,0},X^{1,0}]$ and the assumption that $X$ and $U$ both vanish at a common point in $M$, namely a critical point of $f$.
\end{proof}

The next proposition is the K\"ahler analogue of \cite[Proposition 5.1]{Chodosh:ExpandBrySol}. We use the weaker assumption of positive bisectional curvature (as opposed to positive sectional curvature, as used in \cite[Proposition 5.1]{Chodosh:ExpandBrySol}) in order to apply a barrier argument to the Lichnerowicz PDE.  This proposition will allow us to conclude that the approximate Killing vectors represent actual symmetries.

\begin{proposition}\label{prop:LichPDEbarrier}
Suppose $h$ is a $J$-invariant $(0,2)$-tensor on $M$ which satisfies
$$\Delta_L h + \mathscr{L}_X h - h = 0$$
and $|h| = o(1)$ as $r \to \infty$. Then $h \equiv 0$.
\end{proposition}
\begin{proof}
Since $2\Ric + g \geq g$, one can find $\theta$ large enough such that $\theta(2\Ric + g) \geq h$. Now define
$$\theta_0 = \inf\{\theta \in [0, \infty): \theta(2\Ric + g) - h \geq 0\}.$$
By K\"ahlerity of the metric and the assumptions in the lemma, note that $w := \theta_0 (2\Ric+g) - h$ is $J$-invariant. 

We will show that $\theta_0 > 0$ leads to a contradiction. If $\theta_0 > 0$, by the choice of $\theta_{0}$ and the fact that $|h|=o(1)$, there exists $p \in M$ and real vector $e_1 \in T_p M$ such that $w(e_1, e_1) = 0$ at $p$. Parallel translating $e_1$ in a neighborhood of $p$, we see that the function $w(e_1, e_1)$ has a local minimum at $p$ so $(\Delta w)(e_1, e_1) \geq 0$ and $(D_X w)(e_1, e_1) = 0$ at $p$.

Now we complexify the tangent bundle, and define the (complex) tangent vector field
\begin{equation*}\eta_1 = \frac{1}{2}(e_1 - \sqrt{-1}Je_1)\end{equation*}
which is the $(1,0)$-part of $e_1$. One may easily check that
\begin{align*}
\mathcal{R}: T_p^{1,0}M \times T_p^{1,0}M & \to \C\\
(U, V) & \mapsto \Rm(\eta_1, \overline{U}, \overline{\eta}_1, V)
\end{align*}
is self-adjoint, since
$$\overline{\mathcal{R}(U, V)} = \overline{\Rm(\eta_1, \overline{U}, \overline{\eta}_1, V)} = \Rm(\overline{\eta}_1, U, \eta_1, \overline{V}) = \Rm(\eta_1, \overline{V}, \overline{\eta}_1, U) = \mathcal{R}(V,U).$$
Choose a unitary basis $\{\eta_1, \eta_2, \ldots, \eta_n\}$ for $T^{1,0}_pM$ such that $\Rm(\eta_1, \overline{\eta}_j, \overline{\eta}_1, \eta_k)$ is a diagonal matrix at $p$. We thus have that $\Rm(\eta_1, \overline{\eta}_j, \overline{\eta}_1, \eta_k) = \mu_{j} \delta_{jk}$. Clearly, by positive bisectional curvature, $\mu_j = \Rm(\eta_1, \overline{\eta}_j, \overline{\eta}_1, \eta_j) > 0$ for all $j$.

Evaluating $\Delta_L w + \mathscr{L}_X w - w = 0$ at $p$ and in the $(\eta_i, \overline{\eta}_i)$ direction for any $i \in \{1, \ldots, n\}$ gives
\begin{align*}
0 & = (\Delta w)(\eta_i, \overline{\eta}_i) + 2\sum_{k,l} \Rm(\eta_i, \overline{\eta}_k, \overline{\eta}_i, \eta_l) w(\eta_k, \overline{\eta}_l) - 2 w(\Ric(\eta_i), \overline{\eta}_i)\\
& \quad + (\mathscr{L}_X w) (\eta_i, \overline{\eta}_i) -  w(\eta_i, \overline{\eta}_i)
\end{align*}
Note that
\begin{align*}
(\mathscr{L}_Xw) (\eta_i, \overline{\eta}_i) & = (D_X w) (\eta_i, \overline{\eta}_i) + w(D_{\eta_i} X, \overline{\eta}_i) + w(\eta_i, D_{\overline{\eta}_i}X)\\
& = (D_X w)(\eta_i, \overline{\eta}_i) + 2w(D_{\eta_i}X, \overline{\eta}_i).
\end{align*}
Furthermore, by the soliton equation, we have $2\Ric(\eta_i) = 2 D_{\eta_i}X  - \eta_i$.

Thus, we see that 
\begin{equation}\label{eq:i_ibar}
0 = (\Delta w)(\eta_i, \overline{\eta}_i) + 2\sum_{k,l}\Rm(\eta_i, \overline{\eta}_k, \overline{\eta}_i, \eta_l) w(\eta_k, \overline{\eta}_l) + (D_X w)(\eta_i, \overline{\eta}_i).
\end{equation}
At $p$, taking $i = 1$, we have
$$0 = (\Delta w)(\eta_1, \overline{\eta}_1) + 2\sum_k \mu_k w(\eta_k, \overline{\eta}_k) + 0 \geq  2\sum_k \mu_k w(\eta_k, \overline{\eta}_k).$$
Here we have used the fact that $\Delta w$ and $D_X w$ are $J$-invariant, so that $$(\Delta w)(\eta_1, \overline{\eta}_1) = \frac{1}{4} \left\{(\Delta w)(e_1, e_1) + (\Delta w)(Je_1, Je_1)\right\} \geq 0$$ and similarly for $D_X w$.

Since $\mu_k > 0$ by the positivity of bisectional curvature and $w \geq 0$, we must have $w(\eta_k, \overline{\eta}_k) = 0$ for all $k$ at $p$. Thus, summing \eqref{eq:i_ibar} over $i$, we have
\begin{align*}
\Delta(\Tr \, w) + 2\sum_{k,l}\Ric(\overline{\eta}_k, \eta_l)w(\eta_k, \overline{\eta}_l) + D_X (\Tr \, w) & = 0,\\
 \Delta(\Tr \, w) + D_X (\Tr \, w) = - 2\sum_{k,l}\Ric(\overline{\eta}_k, \eta_l)w(\eta_k, \overline{\eta}_l) & \leq 0.
\end{align*}
by the positivity of the Ricci curvature combined with the non-negativity of $w$.

However, $\Tr \, w = 0$ at $p$ and $\Tr \, w \geq 0$ near $p$ implies $\Tr \, w$ attains an interior local minimum at $p$. By Hopf's strong maximum principle it implies $\Tr \, w \equiv 0$ and hence $w \equiv 0$. However, $\theta_0 (2\Ric + g) \equiv h$ violates the asymptotics of $h$, by positivity of the Ricci curvature. This shows $\theta_ 0 = 0$ and hence $h \leq 0$. Applying the same argument to $-h$ shows that $h \equiv 0$.
\end{proof}

\section{Proof of the Main Result} \label{sec:proof-main-res}
We may now combine the above results to give a proof of Theorem \ref{thm:main}.
\begin{proof}[Proof of Theorem \ref{thm:main}]
Because the K\"ahler cone metric $g_\alpha$ is $U(n)$-rotationally symmetric, we may pick a basis for the corresponding Killing vector fields on $\C^{n}$, which we denote $\{U_a\}_{a=1}^{n^2}$. We note that $[r\dd{r}, U_a] = 0$ for all $a$. This is true because $J(r\dd{r})$ is the Reeb Killing vector field associated to the action $\textup{diag}(e^{\sqrt{-1}\theta},\ldots,e^{\sqrt{-1}\theta})$ in the $U(n)$-group. Clearly this diagonal action commutes with all other $U(n)$-actions, and hence $[J(r\dd{r}), U_a] = 0$. Moreover, the $U_a$'s are real holomorphic hence $[JV, U_a] = J[V, U_a]$ for any vector field $V$. This shows $\left[r\dd{r}, U_a\right] = 0$ for any $a$. We further note that $\left|U_{a} \right|_{g_{\alpha}}= O(r)$

We will now show that these $U_a$ push forward under the biholomorphism $F:\C^{n}\to M$ to Killing vector fields for the soliton metric $g$ (we will use the ``centered'' biholomorphism $F$ constructed in Lemma \ref{lma:centering}). As $F$ is a biholomorphism, $F_* U_a$ is a real holomorphic vector field in $M$ for each $a$. By the asymptotically conical assumption, we see that
\begin{equation*}
|\mathscr{L}_{U_{a}} F^{*}g| = |\mathscr{L}_{U_{a}} k| = o(1)
\end{equation*}
and similarly
\begin{equation*}
\left| \Div \left( \mathscr{L}_{U_{a}}(F^{*}g) \right) -\frac 12 \nabla_{F^{*}g} \left(\tr_{F^{*}g}\  \mathscr{L}_{U_{a}}(F^{*}g)\right)\right| =o(1).
\end{equation*}
Thus, on $M$, this yields $|\mathscr{L}_{F_*U_{a}}(g)| = o(1)$ and 
\begin{equation*}
\left| \Div \left( \mathscr{L}_{F_*U_{a}} g \right) -\frac 12 \nabla \left(\tr\  \mathscr{L}_{ F_*U_{a}}g\right)\right| =o(1).
\end{equation*}
Finally, (2) in Lemma \ref{lma:centering} shows that $F^*X = \frac r 2 \frac{\partial}{\partial r} + Y$ where $Y$ is a vector field so that $|Y| = o(r)$ and $|\nabla Y| = o(1)$. Thus, because the $U_{a}$ were chosen so that their Lie bracket with any purely radial vector field vanishes, we have that $[U_{a},F^*X] = [U_{a},Y] = o(r)$ (using the fact that $|U_{a}| + |\nabla U_{a}|= O(r)$ and the asymptotics of $Y$). 

For simplicity, we denote $\widetilde{U_a} := F_* U_a$. Now, by the computation in \cite[Section 3]{Chodosh:ExpandBrySol}, the above results combine to show that the $\widetilde{U_{a}}$'s satisfy 
\begin{equation*}
\left| \Delta \widetilde{U_{a}} + D_{X} \widetilde{U_{a}} -\frac 12 \widetilde{U_{a}}\right| = o(r). 
\end{equation*}
We may thus apply Lemma \ref{lem:liouville} (using the fact that the $\widetilde{U_{a}}$ clearly vanish at $F(0) \in M$ by how we chose the $U_{a}$ and how we refine $F$ in Lemma \ref{lma:centering}) to conclude that in fact 
\begin{equation*}
 \Delta \widetilde{U_{a}} + D_{X} \widetilde{U_{a}} -\frac 12 \widetilde{U_{a}} = 0.
\end{equation*}
As such, by Proposition \ref{prop:lichPDE}, we see that $h^{(a)}:= \mathscr{L}_{\widetilde{U_{a}}}(g)$ satisfies 
\begin{equation*}
\Delta_{L} h^{(a)} + \mathscr{L}_{X}(h^{(a)}) - h^{(a)}=0.
\end{equation*}
Because we have arranged that $|h^{(a)}| = o(1)$ above and by the fact that $h_a$ is $J$-invariant (since $\widetilde{U_a}$ is real holomorphic), we may thus apply Proposition \ref{prop:LichPDEbarrier} to conclude that $h^{(a)}\equiv 0$. Thus, the $\widetilde{U_{a}}$ are exact Killing vectors for $g$. Finally, we observe that by the computation in Lemma \ref{lem:liouville} we have that $[\widetilde{U_a},X]^{1,0} = 0$. Since $\widetilde{U_a}$ and $X$ are real vector fields, we must have $[\widetilde{U_a},X] = 0$. Thus we have shown that $(M,g,f)$ is $U(n)$-rotationally symmetric, as desired. 
\end{proof}

\section{Further discussion}\label{sec:shrink}
We would like to conclude by briefly discussing rotational symmetry of shrinking K\"ahler-Ricci solitons. A shrinking gradient soliton $(M, g, f)$ is defined by the equation $2\Ric - g = D^2 f$. Rotationally symmetric shrinking K\"ahler-Ricci solitons are constructed in \cite{FeldmanIlmanenKnopf} on the total space of $\mathcal{O}(-k)$-bundles over $\CP^{n-1}$ and they are asymptotically conic at infinity. However, in contrast to the conical expanders constructed in \cite{Cao:KRFsolitons} on $\C^n$, the asymptotic cone angle of these shrinkers is very rigid. In fact, for each pair $(n,k)$, there is at most one $U(n)$-invariant shrinker on $\mathcal{O}(-k)$; as such, the asymptotic cone angle is completely determined by the choice of $(n,k)$. These authors also show that $U(n)$-invariant shrinkers do not exist on $\C^n$. 

We additionally remark that the approach used to establish Theorem \ref{thm:main} does \emph{not} seem to carry over, \textit{mutatis mutandis}, to show a similar result in the shrinking case. In particular, positivity of bisectional curvature is used in a crucial way in the proof of Proposition \ref{prop:LichPDEbarrier}, while by a result of Ni \cite{Ni:AncientSol}, any shrinking K\"ahler-Ricci solitons with positive bisectional curvature must be compact. To the best of authors' knowledge, it is an open problem whether complete asymptotically conical shrinkers must be rotationally symmetric in either the Riemannian or K\"ahler setting. 
\bibliography{bib} 

\providecommand{\bysame}{\leavevmode\hbox to3em{\hrulefill}\thinspace}
\providecommand{\MR}{\relax\ifhmode\unskip\space\fi MR }
\providecommand{\MRhref}[2]{%
  \href{http://www.ams.org/mathscinet-getitem?mr=#1}{#2}
}
\providecommand{\href}[2]{#2}
\begin{thebibliography}{CRW11}

\bibitem[Ban84]{Bando:1984}
Shigetoshi Bando, \emph{On the classification of three-dimensional compact
  {K}aehler manifolds of nonnegative bisectional curvature}, J. Differential
  Geom. \textbf{19} (1984), no.~2, 283--297. \MR{755227 (86i:53042)}

\bibitem[BM13]{BernsteinMettler:2dsol}
Jacob Bernstein and Thomas Mettler, \emph{Two-dimensional gradient ricci
  solitons revisited}, preprint, \url{http://arxiv.org/abs/1303.6854} (2013)

\bibitem[Bre12a]{Brendle:HighDimSoliton}
S.~Brendle, \emph{Rotational symmetry of {R}icci solitons in higher
  dimensions}, preprint, \url{http://arxiv.org/abs/1203.0270} (2012).

\bibitem[Bre12b]{Brendle:3DSolitonUniqueness}
\bysame, \emph{Rotational symmetry of self-similar solutions to the {R}icci
  flow}, to appear in Invent. Math., \url{http://arxiv.org/abs/1202.1264}
  (2012).

\bibitem[Bry]{Bryant:Solitons}
R.L. Bryant, \emph{Ricci flow solitons in dimension three with
  {SO}(3)-symmetries}, available at
  \url{http://www.math.duke.edu/~bryant/3DRotSymRicciSolitons.pdf}.

\bibitem[Cao96]{Cao:ExistenceofKRFsol}
Huai-Dong Cao, \emph{Existence of gradient {K}{\"a}hler-{R}icci solitons},
  Elliptic and parabolic methods in geometry ({M}inneapolis, {MN}, 1994), A K
  Peters, Wellesley, MA, 1996, pp.~1--16. \MR{1417944 (98a:53058)}

\bibitem[Cao97]{Cao:KRFsolitons}
\bysame, \emph{Limits of solutions to the {K}{\"a}hler-{R}icci flow}, J.
  Differential Geom. \textbf{45} (1997), no.~2, 257--272. \MR{1449972
  (99g:53042)}

\bibitem[Cao10]{Cao:RecentSoliton}
\bysame, \emph{Recent progress on {R}icci solitons}, Recent advances in
  geometric analysis, Adv. Lect. Math. (ALM), vol.~11, Int. Press, Somerville,
  MA, 2010, pp.~1--38. \MR{2648937 (2011d:53061)}

\bibitem[CC96]{CheegerColding:Warped}
Jeff Cheeger and Tobias~H. Colding, \emph{Lower bounds on {R}icci curvature and
  the almost rigidity of warped products}, Ann. of Math. (2) \textbf{144}
  (1996), no.~1, 189--237. \MR{1405949 (97h:53038)}

\bibitem[CD11]{ChenDerulle:StructInftyExpandSol}
Chih-Wei Chen and Alix Deruelle, \emph{Structure at infinity of expanding
  gradient {R}icci soliton}, preprint, \url{http://arxiv.org/abs/1108.1468}
  (2011).

\bibitem[Che09]{Chen:RFstrongUnique}
Bing-Long Chen, \emph{Strong uniqueness of the {R}icci flow}, J. Differential
  Geom. \textbf{82} (2009), no.~2, 363--382. \MR{2520796 (2010h:53095)}

\bibitem[Cho13]{Chodosh:ExpandBrySol}
Otis Chodosh, \emph{Expanding {R}icci solitons asymptotic to cones}, preprint,
  \url{http://arxiv.org/abs/1303.2983} (2013).

\bibitem[CRW11]{Cabezas-Rivas:RFviaCGexhaust}
Esther Cabezas-Rivas and Burkhard Wilking, \emph{How to produce a {R}icci flow
  via {C}heeger-{G}romoll exhaustion}, preprint,
  \url{http://arxiv.org/abs/1107.0606} (2011).

\bibitem[CT06]{ChauTam:OnComplexStr}
Albert Chau and Luen-Fai Tam, \emph{On the complex structure of {K}{\"a}hler
  manifolds with nonnegative curvature}, J. Differential Geom. \textbf{73}
  (2006), no.~3, 491--530. \MR{2228320 (2007e:53085)}

\bibitem[CT08]{ChauTam:YauConjSurvey}
\bysame, \emph{A survey on the {K}{\"a}hler-{R}icci flow and {Y}au's
  uniformization conjecture}, Surveys in differential geometry. {V}ol. {XII}.
  {G}eometric flows, Surv. Differ. Geom., vol.~12, Int. Press, Somerville, MA,
  2008, pp.~21--46. \MR{2488949 (2009k:53164)}

\bibitem[CZ06]{ChenZhu:RFunique}
Bing-Long Chen and Xi-Ping Zhu, \emph{Uniqueness of the {R}icci flow on
  complete noncompact manifolds}, J. Differential Geom. \textbf{74} (2006),
  no.~1, 119--154. \MR{2260930 (2007i:53071)}

\bibitem[Der13]{Deruelle:ExpandBrySol}
Alix Deruelle, \emph{Rotational symmetry of non negatively curved expanding
  gradient {R}icci solitons}, preprint, \url{http://arxiv.org/abs/1303.3446}
  (2013).

\bibitem[DW11]{DancerWang:CohomogeneityOne}
Andrew~S. Dancer and McKenzie~Y. Wang, \emph{On {R}icci solitons of
  cohomogeneity one}, Ann. Global Anal. Geom. \textbf{39} (2011), no.~3,
  259--292. \MR{2769300 (2012a:53124)}

\bibitem[FIK03]{FeldmanIlmanenKnopf}
Mikhail Feldman, Tom Ilmanen, and Dan Knopf, \emph{Rotationally symmetric
  shrinking and expanding gradient {K}{\"a}hler-{R}icci solitons}, J.
  Differential Geom. \textbf{65} (2003), no.~2, 169--209. \MR{2058261
  (2005e:53102)}

\bibitem[FW11]{FutakiWang:KRFsolitons}
Akito Futaki and Mu-Tao Wang, \emph{Constructing {K}{\"a}hler-{R}icci solitons
  from {S}asaki-{E}instein manifolds}, Asian J. Math. \textbf{15} (2011),
  no.~1, 33--52. \MR{2786464 (2012e:53077)}

\bibitem[GT11]{GiesenTopping:surfIncompRF}
Gregor Giesen and Peter~M. Topping, \emph{Existence of {R}icci flows of
  incomplete surfaces}, Comm. Partial Differential Equations \textbf{36}
  (2011), no.~10, 1860--1880. \MR{2832165 (2012i:53063)}

\bibitem[Ham95]{Hamitlon:Singularities95}
Richard~S. Hamilton, \emph{The formation of singularities in the {R}icci flow},
  Surveys in differential geometry, {V}ol.\ {II} ({C}ambridge, {MA}, 1993),
  Int. Press, Cambridge, MA, 1995, pp.~7--136. \MR{1375255 (97e:53075)}

\bibitem[Kot06]{Kotschwar:rotsym2dexpand}
Brett Kotschwar, \emph{A note on the uniqueness of complete, positively-curved
  expanding {R}icci solitons in 2-{D}}, available at
  \url{http://math.la.asu.edu/~kotschwar/pub/kotschwar_expsol_11_01.pdf}
  ((2006)).

\bibitem[Mok88]{Mok:Uniformization}
Ngaiming Mok, \emph{The uniformization theorem for compact {K}{\"a}hler
  manifolds of nonnegative holomorphic bisectional curvature}, J. Differential
  Geom. \textbf{27} (1988), no.~2, 179--214. \MR{925119 (89d:53115)}

\bibitem[Mor79]{Mori:FrankelConj}
Shigefumi Mori, \emph{Projective manifolds with ample tangent bundles}, Ann. of
  Math. (2) \textbf{110} (1979), no.~3, 593--606. \MR{554387 (81j:14010)}

\bibitem[Ni05]{Ni:AncientSol}
Lei Ni, \emph{Ancient solutions to {K}{\"a}hler-{R}icci flow}, Math. Res. Lett.
  \textbf{12} (2005), no.~5-6, 633--653. \MR{2189227 (2006i:53097)}

\bibitem[Per]{Perelman:EntropyFunct}
Grisha Perelman, \emph{The entropy formula for the {R}icci flow and its
  geometric applications}, available at
  \url{http://arxiv.org/abs/math/0211159}.

\bibitem[SS10]{SchulzeSimon:Expand-Cones}
Felix Schulze and Miles Simon, \emph{Expanding solitons with non-negative
  curvature operator coming out of cones}, preprint,
  \url{http://arxiv.org/abs/1008.1408} (2010).

\bibitem[SY80]{SiuYau:1980}
Yum~Tong Siu and Shing~Tung Yau, \emph{Compact {K}{\"a}hler manifolds of
  positive bisectional curvature}, Invent. Math. \textbf{59} (1980), no.~2,
  189--204. \MR{577360 (81h:58029)}

\end{thebibliography}
\bibliographystyle{amsalpha}
\end{document}